\documentclass[12pt]{amsart}

\usepackage[utf8]{inputenc}

\usepackage[english]{babel}

\usepackage{amsmath,amsfonts,amssymb,amsthm}

\usepackage{amscd}

\usepackage[a4paper,nomarginpar]{geometry}

\usepackage{epsfig,graphicx,color}

\usepackage[all]{xy}

\newtheorem{theorem}{Theorem}

\theoremstyle{definition}

\newtheorem{remark}[theorem]{Remark}

\newtheorem*{remark*}{Remark}

\newcommand{\N}{\mathbb{N}}  % set of natural numbers
\newcommand{\Z}{\mathbb{Z}}  % set of integer numbers
\newcommand{\C}{\mathbb{C}}  % set of complex numbers
\newcommand{\D}{\mathbb{D}}  % unit disc
\newcommand{\T}{\mathbb{T}}  % unit circle
\newcommand{\eps}{\varepsilon} % abbreviation for epsilon
\newcommand{\Lip}{\operatorname{Lip}}

\allowdisplaybreaks

\begin{document}


\title[A generalized Grobman-Hartman theorem]
      {A generalized Grobman-Hartman theorem}

% Only \author and \address are required; other information is optional.
% Remove any unused author tags.

\author{Nilson C. Bernardes Jr.}
\address{Departamento de Matem\'atica Aplicada, Instituto de Matem\'atica,
    Universidade Federal do Rio de Janeiro, Caixa Postal 68530,
    Rio de Janeiro, RJ, 21945-970, Brazil.}
\curraddr{}
\email{ncbernardesjr@gmail.com}
\thanks{}

\author{Ali Messaoudi}
\address{Departamento de Matem\'atica, Universidade Estadual Paulista,
    Rua Crist\'ov\~ao Colombo, 2265, Jardim Nazareth, S\~ao Jos\'e do
    Rio Preto, SP, 15054-000, Brasil.}
\curraddr{}
\email{ali.messaoudi@unesp.br}
\thanks{}



\date{} %\today

\dedicatory{}

\begin{abstract}
We prove that any generalized hyperbolic operator on any Banach space
is structurally stable. As a consequence, we obtain a generalization
of the classical Grobman-Hartman theorem.
\end{abstract}

\maketitle

%%%%%%%%%%%%%%%%%%%%%%%%%%%%%%%%%%%%%%%%%%%%%%%%%%%%%%%%%%%%%%%%%%%%%%%%%%%%%

\section{Introduction}

Let $E$ be a metric space. Recall that two continuous maps
$\varphi, \, \psi : E \to E$ are {\em topologically conjugate} if there is
a homeomorphism $h : E \to E$ such that $h \circ \varphi = \psi \circ h$.
Given a Banach space $X$, an invertible bounded linear operator $T$ on $X$
is said to be {\em structurally stable} if there exists $\eps > 0$ such that
$T + \varphi$ is topologically conjugate to $T$ whenever $\varphi : X \to X$
is a Lipschitz map with norm
$\|\varphi\|_\infty = \sup_{x \in X} \|\varphi(x)\| \leq \eps$
and Lipschitz constant
$\Lip(\varphi) = \sup_{x \neq y} \frac{\|\varphi(x)-\varphi(y)\|}{\|x-y\|}
               \leq \eps$.
Structural stability is a fundamental notion in the area of dynamical systems.
It was introduced by Andronov and Pontrjagin \cite{AAndLPon37} for a certain
class of differentiable flows on the plane. Nowadays, there are many
variations of this notion in different contexts. In the definition of
structural stability it is usual to consider $C^1$ pertubations with small
$C^1$ norm. Here we are following Pugh \cite{CPug69}, where it is allowed
more general pertubations, namely Lipschitz pertubations with small Lipschitz
norm. We refer the reader to \cite{AKatBHas95,JRob72,MShu87} for nice
expositons about structural stability.

\smallskip
Another fundamental notion in the area of dynamical systems is that of
{\em hyperbolicity}. Recall that a bounded linear operator $T$ on a complex
Banach space $X$ is said to be {\em hyperbolic} if its spectrum $\sigma(T)$
does not intersect the unit circle $\T$ in the comple plane. In the case of
real Banach spaces, it is required that $\sigma(T_\C) \cap \T = \emptyset$,
where $T_\C$ denotes the complexification of $T$. It is well-known that
$T$ is hyperbolic if and only if there are an equivalent norm $\|\cdot\|$
on $X$ and a splitting $X = X_s \oplus X_u$, $T = T_s \oplus T_u$
(the {\em hyperbolic splitting} of $T$), where $X_s$ and $X_u$ are closed
$T$-invariant subspaces of $X$ (the {\em stable} and the {\em unstable
subspaces} for $T$), $T_s = T|_{X_s}$ is a proper contraction
(i.e., $\|T_s\| < 1$), $T_u = T|_{X_u}$ is invertible and is a proper dilation
(i.e., $\|T_u^{-1}\| < 1$), and the identification of $X$ with the product
$X_s \times X_u$ identifies $\|\cdot\|$ with the max norm on the product.

\medskip
Let us recall the following classical result from the 1960's.

\bigskip
\noindent
{\bf Theorem A.} {\it Every invertible hyperbolic operator on a Banach
space is structurally stable}.

\bigskip
This result was originally obtained by Hartman \cite{PHar60} for operators
on finite-dimensional euclidean spaces. The general case was independently
obtained by Palis \cite{JPal68} and Pugh \cite{CPug69}, motivated by an
argument in Moser \cite{JMos69}.

\smallskip
It is natural to ask: {\it Does the converse of Theorem~A hold?}
It was soon realized that the answer is ``yes'' in the finite-dimensional
setting. Indeed, the 1972 paper \cite{JRob72} by Robbin already contains
a proof of this fact. However, the full question was answered only very
recently by Bernardes and Messaoudi \cite{NBerAMes}. In fact, in this paper
it was characterized the invertible weighted shifts on the spaces
$\ell_p(\Z)$ ($1 \leq p < \infty$) and $c_0(\Z)$ that have the shadowing
property and it was proved that all of them are structurally stable;
as a consequence, examples of structurally stable operators that are not
hyperbolic were obtained, answering the above question in the negative.
More precisely, the following result is contained in \cite{NBerAMes}.

\bigskip
\noindent
{\bf Theorem B.} {\it Let $Y = \ell_p(\Z)$ $(1 \leq p < \infty)$ or
$Y = c_0(\Z)$. Let $w = (w_n)_{n \in \Z}$ be a bounded sequence of scalars
with $\inf_{n \in \Z} |w_n| > 0$ and consider the bilateral weighted
backward shift
$$
B_w : (x_n)_{n \in \Z} \in Y \mapsto (w_{n+1}x_{n+1})_{n \in \Z} \in Y.
$$
If
\begin{equation}\label{Ineq}
\lim_{n \to \infty} \sup_{k \in \N}
  |w_{-k} w_{-k-1} \cdot \ldots \cdot w_{-k-n}|^\frac{1}{n} < 1
\ \text{ and } \
\lim_{n \to \infty} \inf_{k \in \N}
  |w_k w_{k+1} \cdot \ldots \cdot w_{k+n}|^\frac{1}{n} > 1,
\end{equation}
then $B_w$ is structurally stable and not hyperbolic.}

\bigskip
In the present work we will obtain a result that unifies Theorems~A and~B.
In order to be more precise, recall that an invertible bounded linear
operator $T$ on $X$ is said to be {\em generalized hyperbolic} if we can
write
\begin{equation}\label{eq1}
X = M \oplus N,
\end{equation}
where $M$ and $N$ are closed subspaces of $X$ such that $T(M) \subset M$,
$T^{-1}(N) \subset N$,
\begin{equation}\label{eq2}
\sigma(T|_M) \subset \D \ \ \text{ and } \ \ \sigma(T^{-1}|_N) \subset \D,
\end{equation}
where $\D$ denotes the open unit disc in the complex plane.
This class of operators was introduced by Bernardes et al.\
\cite{BerCirDarMesPuj18}, where it was proved that each element of this
class has the shadowing property. But the terminology ``generalized
hyperbolic'' was given by Cirilo et al.\ \cite{CirGolPuj}, where it was
proved that this class is open in the space of all invertible bounded linear
operators. It is clear that this class contains the invertible hyperbolic
operators. It also contains the invertible weighted shifts from Theorem~B.
In order to see this, it is enough to consider
$$
M = \{(x_n)_{n \in \Z} \in Y : x_n = 0 \text{ for all } n > 0\},
$$
$$
N = \{(x_n)_{n \in \Z} \in Y : x_n = 0 \text{ for all } n \leq 0\},
$$
and to observe that the spectral radius formula shows that the estimates
in (\ref{Ineq}) give the inclusions in (\ref{eq2}). We will prove in
Section~\ref{GH} that every generalized hyperbolic operator on a Banach
space is structurally stable, which unifies Theorems~A and~B.
The class of generalized hyperbolic operators contains all the structurally
stable operators that are known up to now. It is an open problem whether
or not every structurally stable operator lies in this class.

\smallskip
The classical {\em Grobman-Hartman theorem} asserts that if $p$ is a
hyperbolic fixed point of a $C^1$ diffeomorphism $F$ on a Banach space $X$,
then there is a neighborhood of $p$ where $F$ is topologically conjugate to
its derivative at $p$. This {\em linearization theorem} was independently
obtained by Grobman \cite{DGro62} (announced in \cite{DGro59}) and Hartman
\cite{PHar60,PHar63} in the finite-dimensional setting. The extension
to Banach spaces is due independently to Palis \cite{JPal68} and Pugh
\cite{CPug69}. The Grobman-Hartman theorem plays a major role in the areas
of dynamical systems and differential equations. We refer the reader to
\cite{JGucPHol83,AKatBHas95,WMelJPal82,MShu87} for more details on this
important theorem and its applications.

\smallskip
In Section \ref{Grobman-Hartman} we will obtain a generalization of the
Grobman-Hartman theorem by showing that we can replace the hyperbolicity
hypothesis on the fixed point by generalized hyperbolicity. Moreover,
we will prove that, even in this more general case, the homeomorphism
conjugating the map and its derivative at the fixed point can be
chosen to be $\theta$-H\"older (for suitable values of $\theta$) near the
fixed point.

%%%%%%%%%%%%%%%%%%%%%%%%%%%%%%%%%%%%%%%%%%%%%%%%%%%%%%%%%%%%%%%%%%%%%%%%%%%%%

\section{Generalized hyperbolic operators are structurally stable}\label{GH}

Given Banach spaces $X$ and $Y$, we denote by $U_b(X;Y)$ the Banach space of
all bounded uniformly continuous maps $\varphi : X \to Y$ endowed with the
supremum norm. In the case $X = Y$, we write $U_b(X)$ instead of $U_b(X;X)$.

\smallskip
Our goal in this section is to prove that generalized hyperbolicity implies
structural stability. Actually, we will obtain a formally stronger property,
namely: strong structural stability. Recall that an invertible bounded linear
operator $T$ on a Banach space $X$ is said to be {\em strongly structurally
stable} if for every $\gamma > 0$ there exists $\eps > 0$ such that the
following property holds: for any Lipschitz map $\varphi \in U_b(X)$ with
$\|\varphi\|_\infty \leq \eps$ and $\Lip(\varphi) \leq \eps$,
there is a homeomorphism $h : X \to X$ such that
$h \circ T = (T + \varphi) \circ h$ and $\|h - I\|_\infty \leq \gamma$.
So, it is now required that the homeomorphism $h$ conjugating $T$ and
$T + \varphi$ is close to the identity operator. Although this notion
is formally stronger than structural stability, it is still an open problem
whether or not these two notions are equivalent.

\begin{theorem}\label{t1}
Every generalized hyperbolic operator on a Banach space is strongly
structurally stable.
\end{theorem}

\begin{remark}
An important difference between the proofs of Theorem~\ref{t1} and Theorem~A
is that in case where the operator is generalized hyperbolic and not
hyperbolic, the conjugation $H$ is not unique and we have to choose
$H = Id+h$ where $h$ belongs to an adequate space of functions.
\end{remark}

\begin{proof}
Let $T$ be a generalized hyperbolic operator on a Banach space $X$ and
let $M$ and $N$ be as in (\ref{eq1}) and (\ref{eq2}). Let
$$
P_M : X \to M \ \ \ \text{ and } \ \ \ P_N : X \to N
$$
be the projections associated to the decomposition of $X$ given
by (\ref{eq1}), and put
$$
d = \max\{\|P_M\|,\|P_N\|\}.
$$
By (\ref{eq2}) and the spectral radius formula, there are constants
$c \geq 1$ and $0 < t < 1$ such that
\begin{equation}\label{eq3}
\|T^ny\| \leq c\,t^n\|y\| \ \text{ and } \ \|T^{-n}z\| \leq c\,t^n\|z\|
\ \text{ whenever } n \in \N_0, y \in M \text{ and } z \in N.
\end{equation}
Consider the closed subspace $Y = M + T^{-1}(N)$ of $X$.
In order to prove that $T$ is strongly structurally stable, we fix
$0 < \gamma < 1$ and put
$$
\eps = \frac{\gamma\,(1-t)}{c\,d\,(1+t)}\cdot
$$
Let $\beta \in U_b(X)$ be a Lipschitz map with $\|\beta\|_\infty \leq \eps$
and $\Lip(\beta) \leq \eps$. Put $S = T + \beta$. We have to find a
homeomorphism $H : X \to X$ such that $H \circ T = S \circ H$ and
$\|H - I\|_\infty \leq \gamma$. Actually, our $H$ will be a uniform
homeomorphism of the form $H = I + h$, where $h \in U_b(X;Y)$ and
$\|h\|_\infty \leq \gamma$. We divide the remaining of the proof in
five steps.

\medskip
\noindent {\bf Step 1.} {\it For any uniform homeomorphism $R : X \to X$,
the bounded linear map
$$
\Psi: \varphi \in U_b(X;Y) \mapsto \varphi \circ R - T \circ \varphi \in U_b(X)
$$
is bijective. Moreover, its inverse is given by
\begin{equation}\label{eq4}
\Psi^{-1}(\alpha)(x) = \sum_{k=0}^\infty T^k P_M(\alpha(R^{-k-1}x))
                     - \sum_{k=1}^\infty T^{-k} P_N(\alpha(R^{k-1}x)).
\end{equation}
In particular,}
\begin{equation}\label{eq5}
\|\Psi^{-1}(\alpha)\|_\infty \leq \frac{c\,d\,(1+t)}{1-t}\, \|\alpha\|_\infty.
\end{equation}

\medskip
Indeed, fix $\alpha \in U_b(X)$ and suppose that $\varphi \in U_b(X;Y)$
satisfies $\Psi(\varphi) = \alpha$, that is,
\begin{equation}\label{eq6}
\varphi(Rx) - T(\varphi(x)) = \alpha(x) \ \ \text{ for all } x \in X.
\end{equation}
Then, a simple induction argument shows that
\begin{equation}\label{eq7}
\varphi(R^nx) = T^n(\varphi(x)) + \sum_{k=1}^n T^{n-k}(\alpha(R^{k-1}x))
  \ \ \text{ for all } n \in \N.
\end{equation}
By applying $T^{-n}$ to both sides of the above equality, we obtain
\begin{equation}\label{eq8}
\varphi(x) = T^{-n}(\varphi(R^nx)) - \sum_{k=1}^n T^{-k}(\alpha(R^{k-1}x))
           = y_n(x) + z_n(x),
\end{equation}
where
\begin{align*}
y_n(x) &= T^{-n}P_M(\varphi(R^nx)) - \sum_{k=1}^n T^{-k}P_M(\alpha(R^{k-1}x)),\\
z_n(x) &= T^{-n}P_N(\varphi(R^nx)) - \sum_{k=1}^n T^{-k}P_N(\alpha(R^{k-1}x)).
\end{align*}
It is clear that $z_n(x) \in T^{-1}(N)$ for all $n \in \N$ and $x \in X$.
We claim that $y_n(x) \in M$ for all $n$ and $x$. For this purpose, write
$$
y_n(x) = a_n(x) + b_n(x)
         \ \text{ with } a_n(x) \in M \text{ and } b_n(x) \in N.
$$
Consider the case $n = 1$. We have that
$$
T(y_1(x)) = P_M(\varphi(Rx)) - P_M(\alpha(x)) \in M,
$$
and so $T(b_1(x)) = T(y_1(x)) - T(a_1(x)) \in M$ as well.
Since $\varphi(x) = a_1(x) + (b_1(x)+z_1(x))$, $a_1(x) \in M$ and
$b_1(x)+z_1(x) \in N$, the fact that $\varphi(x) \in Y$ implies that
$b_1(x)$ must belong to the set $T^{-1}(M) \cap T^{-1}(N) = \{0\}$, that is,
$y_1(x) \in M$. Now, suppose that for a certain $n \geq 1$, we have that
$y_n(x) \in M$ for all $x$. Then,
$$
T(y_{n+1}(x)) = y_n(Rx) - P_M(\alpha(x)) \in M.
$$
By arguing as above, we conclude that $b_{n+1}(x) = 0$, that is,
$y_{n+1}(x) \in M$. By induction, our claim is proved.
Thus, (\ref{eq8}) gives $P_N(\varphi(x)) = z_n(x)$ for all $n \in \N$.
Since $T^{-n}P_N(\varphi(R^nx)) \to 0$ as $n \to \infty$, we obtain
\begin{equation}\label{eq9}
P_N(\varphi(x)) = - \sum_{k=1}^\infty T^{-k}P_N(\alpha(R^{k-1}x))
  \ \ \ (x \in X).
\end{equation}
Now, if we apply (\ref{eq7}) with $R^{-n}x$ in place of $x$, we get
\begin{equation}\label{eq10}
\varphi(x) = T^n(\varphi(R^{-n}x)) + \sum_{k=0}^{n-1} T^k(\alpha(R^{-k-1}x))
           = y'_n(x) + z'_n(x),
\end{equation}
where
\begin{align*}
y'_n(x) &= T^nP_M(\varphi(R^{-n}x)) + \sum_{k=0}^{n-1} T^kP_M(\alpha(R^{-k-1}x)),\\
z'_n(x) &= T^nP_N(\varphi(R^{-n}x)) + \sum_{k=0}^{n-1} T^kP_N(\alpha(R^{-k-1}x)).
\end{align*}
It is clear that $y'_n(x) \in M$ for all $n \in \N$ and $x \in X$.
We claim that $z'_n(x) \in T^{-1}(N)$ for all $n$ and $x$.
Indeed, consider the case $n = 1$. Since $\varphi(R^{-1}x) \in Y$,
we can write $\varphi(R^{-1}x) = a + b$ with $a \in M$ and
$b \in T^{-1}(N)$. Hence,
$$
z'_1(x) = TP_N(\varphi(R^{-1}x)) + P_N(\alpha(R^{-1}x))
        = Tb + P_N(\alpha(R^{-1}x)) \in N,
$$
which implies that $z'_1(x) \in T^{-1}(N)$, because $\varphi(x) \in Y$.
Now, suppose that for a certain $n \geq 1$, we have that
$z'_n(x) \in T^{-1}(N)$ for all $x$. Then,
$$
z'_{n+1}(x) = T(z'_n(R^{-1}x)) + P_N(\alpha(R^{-1}x)) \in N,
$$
which implies that $z'_{n+1}(x) \in T^{-1}(N)$, since $\varphi(x) \in Y$.
This proves our second claim.
Hence, (\ref{eq10}) gives $P_M(\varphi(x)) = y'_n(x)$ for all $n \in \N$.
Since $T^nP_M(\varphi(R^{-n}x)) \to 0$ as $n \to \infty$, we obtain
\begin{equation}\label{eq11}
P_M(\varphi(x)) = \sum_{k=0}^\infty T^kP_M(\alpha(R^{-k-1}x))
  \ \ \ (x \in X).
\end{equation}
By (\ref{eq9}) and (\ref{eq11}), $\varphi$ must be unique. On the other hand,
the estimates
\begin{equation}\label{eq12}
\sum_{k=1}^\infty \|T^{-k}P_N(\alpha(R^{k-1}x))\|
  \leq  \frac{c\,d\,t}{1-t}\, \|\alpha\|_\infty,
\end{equation}
\begin{equation}\label{eq13}
\sum_{k=0}^\infty \|T^kP_M(\alpha(R^{-k-1}x))\|
  \leq \frac{c\,d}{1-t}\, \|\alpha\|_\infty
\end{equation}
show that the series in (\ref{eq9}) and (\ref{eq11}) converge absolutely
and uniformly on $X$. Therefore, if we define $\varphi : X \to Y$
by means of equations (\ref{eq9}) and (\ref{eq11}), we obtain a map
$\varphi \in U_b(X;Y)$. Moreover, an easy computation shows that (\ref{eq6})
holds, that is, $\Psi(\varphi) = \alpha$. This shows that $\Psi$ is bijective
and that $\Psi^{-1}$ is given by (\ref{eq4}). Finally, the estimate
(\ref{eq5}) follows immediately from the estimates (\ref{eq12}) and
(\ref{eq13}), which completes the proof of Step~1.

\medskip
\noindent {\bf Step 2.} {\it There is a unique $h \in U_b(X;Y)$ such that
the uniformly continuous map $H = I+h : X \to X$ satisfies
\begin{equation}\label{eq14}
H \circ T = S \circ H.
\end{equation}
Moreover,}
\begin{equation}\label{eq15}
\|H - I\|_\infty = \|h\|_\infty \leq \gamma.
\end{equation}

\medskip
We apply Step~1 with $R = T$ to obtain the linear isomorphism
$$
\Psi_1 : \varphi \in U_b(X;Y) \mapsto \varphi \circ T - T \circ \varphi
         \in U_b(X).
$$
Since (\ref{eq14}) is equivalent to $h = \Psi_1^{-1}(\beta \circ (I+h))$,
we have that $h \in U_b(X;Y)$ has the desired property if and only if
it is a fixed point of the map
$$
\Phi_1 : \varphi \in U_b(X;Y) \mapsto \Psi_1^{-1}(\beta \circ (I+\varphi))
         \in U_b(X;Y).
$$
Since $\Phi_1$ is Lipschitz with
$$
\Lip(\Phi_1) \leq \|\Psi_1^{-1}\| \Lip(\beta)
             \leq \frac{c\,d\,(1+t)}{1-t}\, \eps < 1
$$
(where we have used (\ref{eq5}) and our choice of $\eps$), the existence
and uniqueness of $h$ follows from Banach's contraction principle. Moreover,
$$
\|h\|_\infty = \|\Phi_1(h)\|_\infty \leq \|\Psi_1^{-1}\| \|\beta\|_\infty
             \leq \frac{c\,d\,(1+t)}{1-t}\, \eps = \gamma,
$$
which gives (\ref{eq15}).

\medskip
\noindent {\bf Step 3.} {\it There is a unique $h' \in U_b(X;Y)$ such that
the uniformly continuous map $H' = I+h' : X \to X$ satisfies}
\begin{equation}\label{eq16}
H' \circ S = T \circ H'.
\end{equation}

\medskip
We apply Step~1 with $R = S$ to obtain the linear isomorphism
$$
\Psi_2 : \varphi \in U_b(X;Y) \mapsto \varphi \circ S - T \circ \varphi
         \in U_b(X).
$$
In this case, a simple computation shows that (\ref{eq16}) is equivalent to
$\Psi_2(h') = -\beta$. Thus, $h' = \Psi_2^{-1}(-\beta)$ is the only solution.

\medskip
\noindent {\bf Step 4.} $H' \circ H = I$.

\medskip
By (\ref{eq14}) and (\ref{eq16}),
\begin{equation}\label{eq17}
H' \circ H \circ T = T \circ H' \circ H.
\end{equation}
Since $H' \circ H = I + u$ with $u \in U_b(X;Y)$, (\ref{eq17}) gives 
$\Psi_1(u) = 0$, and so $u = 0$.

\medskip
\noindent {\bf Step 5.} $H \circ H' = I$.

\medskip
By (\ref{eq14}) and (\ref{eq16}),
\begin{equation}\label{eq18}
H \circ H' \circ S = S \circ H \circ H'.
\end{equation}
Write $H \circ H' = I + v$ with $v \in U_b(X;Y)$. A simple computation
shows that (\ref{eq18}) is equivalent to
$v = \Psi_2^{-1}(\beta \circ (I+v) - \beta)$, that is, $v$ is a fixed point
of the map
$$
\Phi_2 : \varphi \in U_b(X;Y) \mapsto \Psi_2^{-1}((\beta \circ (I+\varphi)) - \beta)
         \in U_b(X;Y).
$$
As before, $\Phi_2$ is a contraction, and so $\Phi_2$ has $v$ as its unique
fixed point. However, now we have $\Phi_2(0) = 0$. Thus, $v = 0$, as was
to be shown.

\medskip
Finally, Steps 4 and 5 show that $H$ is a uniform homeomorphism, completing
the proof that $T$ is strongly structurally stable.
\end{proof}

%%%%%%%%%%%%%%%%%%%%%%%%%%%%%%%%%%%%%%%%%%%%%%%%%%%%%%%%%%%%%%%%%%%%%%%%%%%%%

\section{A generalized Grobman-Hartman theorem}\label{Grobman-Hartman}

Let $X$ be a Banach space and $F: X \to X$ be a differentiable map.
Suppose that $p$ is a fixed point of $F$.
We say that $p$ is a {\em generalized hyperbolic fixed point} of $F$ if
the derivative $DF_p$ of $F$ at $p$ is a generalized hyperbolic operator
on $X$, that is, there is a splitting
$$
X = E_p^- \oplus E_p^+,
$$
where $E_p^-$ and $E_p^+$ are closed subspaces of $X$ with
$DF_p(E_p^+) \subset E_p^+$, $(DF_p)^{-1}(E_p^-) \subset E_p^-$,
\begin{equation}\label{equa2}
\sigma(DF_p|_{E_p^+}) \subset \D \ \ \text{ and } \ \
\sigma((DF_p)^{-1}|_{E_p^-}) \subset \D.
\end{equation}

\smallskip
Let $U$ and $V$ be open subsets of $X$ and let $\theta > 0$.
Recall that a homeomorphism $H : U \to V$ is said to be
{\em $\theta$-H\"older} if there is a constant $c > 0$ such that
$$
\|H(x) - H(x')\| \leq c\, \|x - x'\|^\theta
\ \text{ for all } x,x' \in U
$$
and
$$
\|H^{-1}(y) - H^{-1}(y')\| \leq c\, \|y - y'\|^\theta
\ \text{ for all } y,y' \in V.
$$

\smallskip
As a consequence of Theorem \ref{t1}, we will now obtain a generalization
of the Grobman-Hartman theorem to the case of generalized hyperbolic
fixed points of $C^1$ diffeomorphisms on Banach spaces. We will also
show that the linearization can be chosen to be $\theta$-H\"older near
the fixed point, provided $\theta > 0$ is small enough.

\begin{theorem}\label{t2}
Let $X$ be a Banach space and $F: X \to X$ be a $C^1$ diffeomorphism.
If $p$ is a generalized hyperbolic fixed point of $F$, then $F$ is
topologically conjugated to $DF_p$ near $p$, that is, there exist
a homeomorphism $H : X \to X$ and an open neighborhood $U$ of $p$ in $X$
such that
$$
H \circ F = DF_p \circ H \ \text{ on } U.
$$
Moreover, for $\theta > 0$ small enough, we have that $U$ and $H$ can be
chosen so that $H : U \to H(U)$ is a $\theta$-H\"older homeomorphism.
\end{theorem}

\begin{proof}
We first assume that $p = 0$. Put $T = DF_0$ and $\alpha = F - T$.
We have that $\alpha(0) = F(0) = 0$ and $D\alpha_0 = 0$.
By Theorem~\ref{t1} and the hypothesis on the fixed point $p$, the operator
$T$ is structurally stable. Hence, there exists $0 < \eps < \|T^{-1}\|^{-1}$
such that $T + \varphi$ is topologically conjugate to $T$ whenever
$\varphi : X \to X$ is a Lipschitz map with $\|\varphi\|_\infty \leq \eps$
and $\Lip(\varphi) \leq \eps$. Let $U$ be an open neighborhood of $0$ in $X$
such that $\Lip(\alpha|_U) < \frac{\eps}{3}\cdot$ By a classical result (see,
for example, Lemma~2 of \cite{JPal68}), there exists a bounded Lipschitz map
$\beta: X \to X$ such that $\beta|_U = \alpha|_U$,
$\|\beta\|_\infty \leq \eps$ and $\Lip(\beta) \leq \eps$.
Let $H : X \to X$ be a homeomorphism such that
$H \circ (T + \beta) = T \circ H$.
Then,
$$
H \circ F = DF_0 \circ H \ \text{ on } U,
$$
which proves the first assertion in Theorem~\ref{t2}. In order to prove
the second assertion, let $M = E_0^+$, $N = E_0^-$, $Y = M + T^{-1}(N)$,
$S = T + \beta$ and
$$
\Psi : \varphi \in U_b(X;Y) \mapsto \varphi \circ S - T \circ \varphi \in U_b(X).
$$
By Step~3 in the proof of Theorem~\ref{t1}, we can assume that $H$ has the
form $H = I + h$, where $h \in U_b(X;Y)$ is given by $h = \Psi^{-1}(-\beta)$.
Hence, by (\ref{eq4}),
\begin{equation}\label{Formula}
h(x) = - \sum_{k=0}^\infty T^kP_M(\beta(S^{-k-1}x)) +
         \sum_{k=1}^\infty T^{-k}P_N(\beta(S^{k-1}x)).
\end{equation}
By renorming $X$, if necessary, we may assume that
$$
\|T|_M\| < 1 \ \ \ \text{ and } \ \ \ \|T^{-1}|_N\| < 1.
$$
We assume $M \neq \{0\}$ and $N \neq \{0\}$, leaving the other cases to the
reader. Let
$$
0 < \theta < \min\Big\{-\frac{\ln \|T^{-1}|_N\|}{\ln \|T\|},\;
-\frac{\ln \|T|_M\|}{\ln \|T^{-1}\|}\Big\} \leq 1.
$$
Then
\begin{equation}\label{cara}
\max\big\{\|T|_M\| \|T^{-1}\|^\theta,\|T^{-1}|_N\| \|T\|^\theta\big\} < 1.
\end{equation}
Since $\beta : X \to X$ is a bounded Lipschitz map and $\theta \in (0,1]$,
$\beta$ is $\theta$-H\"older. More precisely, since
$\|\beta\|_\infty \leq \eps$ and $\Lip(\beta) \leq \eps$, we have that
\begin{equation}\label{beta}
\|\beta(x) - \beta(y)\| \leq 2\eps \|x - y\|^\theta
\ \ \text{ for all } x,y \in X.
\end{equation}
On the other hand,
\begin{equation}\label{Ineq1}
\|S^nx - S^ny\| \leq (\|T\| + \eps)^n \|x - y\| \ \ \text{ for all } x,y \in X.
\end{equation}
Since $S^{-1}x = T^{-1}x - T^{-1}(\beta(S^{-1}x))$, we deduce that
$$
\|S^{-1}x - S^{-1}y\| \leq \frac{\|T^{-1}\|}{1 - \|T^{-1}\| \eps}\,\|x - y\|
  = (\|T^{-1}\| + \eps\, s) \|x-y\|,
$$
where $s = \|T^{-1}\|^2/(1 - \|T^{-1}\| \eps) > 0$. Thus,
\begin{equation}\label{Ineq2}
\|S^{-n}x - S^{-n}y\| \leq (\|T^{-1}\| + \eps\, s)^n \|x - y\|
  \ \ \text{ for all } x,y \in X.
\end{equation}
Now, by using (\ref{Formula}), (\ref{beta}), (\ref{Ineq1}) and (\ref{Ineq2}),
we obtain
\begin{align*}
\|h(x) - h(y)\| &\leq
      \sum_{k=0}^\infty \|T^k P_M(\beta(S^{-k-1}x) - \beta(S^{-k-1}y))\|\\
& \ \ \ \ + \sum_{k=1}^\infty \|T^{-k} P_N(\beta(S^{k-1}x) - \beta(S^{k-1}y))\|\\
&\leq C \|x - y\|^\theta,
\end{align*}
where
$$
C = 2\eps \|P_M\| \sum_{k=0}^\infty \|T|_M\|^k (\|T^{-1}\| + \eps\, s)^{(k+1) \theta}
  + 2\eps \|P_N\| \sum_{k=1}^\infty \|T^{-1}|_N\|^k (\|T\| + \eps)^{(k-1)\theta}.
$$
It follows from (\ref{cara}) that $C$ is a finite constant provided we choose
$\eps > 0$ small enough. Hence, the map $h : X \to Y$ is $\theta$-H\"older.
We know from the proof of Theorem~\ref{t1} that $H^{-1} = I + h'$,
where $h'$ is of the same type as $h$. Thus, the map $h' : X \to Y$ is also
$\theta$-H\"older. By choosing $U$ so that both $U$ and $V = H(U)$ have
diameters $< 1$, we conclude that $H : U \to V$ is a $\theta$-H\"older
homeomorphism.

Now, suppose that $p \neq 0$ and consider the $C^1$ diffeomorphism
$G: X \to X$ defined by
$$
G(x) = F(x+p) - p \ \ \text{ for all } x \in X.
$$
Since $G(0) = 0$ and $DG_0 = DF_p$, there exist a homeomorphism
$K : X \to X$ and an open neighborhood $V$ of $0$ in $X$ such that
$$
K \circ G = DG_0 \circ K \ \text{ on } V.
$$
Consider the open neighborhood $U = p + V$ of $p$ in $X$ and the
homeomorphism $H : X \to X$ given by $H(y) = K(y - p)$ for all $y \in X$.
Then,
$$
H \circ F = DF_p \circ H \ \text{ on } U.
$$
Moreover, $H$ can be chosen to be $\theta$-H\"older on $U$ for $\theta > 0$
small enough.
\end{proof}

\begin{remark}
The fact that the linearization can be chosen to be locally $\theta$-H\"older
for small enough $\theta$ was proved in the case of a hyperbolic fixed point
in \cite{BR}.
\end{remark}

We close this work by proposing the following open problem:
{\it Does every infinite-dimensional (separable) Banach space support a
nonhyperbolic (strongly) structurally stable operator?}

%%%%%%%%%%%%%%%%%%%%%%%%%%%%%%%%%%%%%%%%%%%%%%%%%%%%%%%%%%%%%%%%%%%%%%%%%%%%%

\section*{Acknowledgement}

The first author was partially supported by project {\#}304207/2018-7
of CNPq (National Council for Scientific and Technological Development
-- Brazil) and by grant {\#}2017/22588-0 of FAPESP (S\~ao Paulo Research
Foundation -- Brazil).
The second author was partially supported by project {\#}307776/2015-8
of CNPq and by projects {\#}2013/24541-0 and {\#}2017/22588-0 of FAPESP.

%%%%%%%%%%%%%%%%%%%%%%%%%%%%%%%%%%%%%%%%%%%%%%%%%%%%%%%%%%%%%%%%%%%%%%%%%%%%%

\end{document}